\documentclass{article}   	
\usepackage{geometry}                		
\usepackage{graphicx}				
\usepackage{amssymb, amsmath, amscd, amsthm}
\usepackage{caption}
\usepackage[usenames]{color}   
\usepackage{pgfplots}  
\definecolor{plum}{RGB}{102,0,204}

\newtheorem*{main}{The Main Theorem}
\newtheorem{prop}{Proposition}
\newtheorem{lemma}{Lemma}

\newcommand*\circled[1]{\tikz[baseline=(char.base)]{
            \node[shape=circle,draw,inner xsep=2pt, inner ysep=0pt,minimum size=2.2em,scale=.75] (char) {#1};}}
\newcommand*\block[1]{\tikz[baseline=(char.base)]{
            \node[shape=rectangle,draw,inner xsep=2pt, inner ysep=0pt] (char) {#1};}}
            
\def\N{{\mathbb N}}
\def\Z{{\mathbb Z}}
\def\md{{~(\operatorname{mod}~ 4)}}
\def\Id{{\operatorname{Id}}}

\title{How do you fix an oval track puzzle?}
\author{David A. Nash and Sara Randall} 
\date{}							

\begin{document}
\maketitle
\begin{abstract}
The oval track group, $OT_{n,k}$, is the subgroup of the symmetric group, $S_n$ generated by the basic moves in a generalized oval track puzzle with $n$ tiles and a turntable of size $k$.  In this paper we completely describe the oval track group for all possible $n$ and $k$ and use this information to answer the following question: If the tiles are removed from an oval track puzzle, how must they be returned in order to ensure that the puzzle is still solvable?  As part of this discussion we introduce the parity subgroup of $S_n$ in the case when $n$ is even.
\end{abstract}

\section{Introduction.}\label{sec:Intro}

The Top Spin puzzle was invented by Binary Arts (now Think Fun) in 1989.  The game consists of an oval track containing twenty tiles, numbered 1 through 20, and an intersecting turntable which holds 4 tiles at a time (see Figure~\ref{fig:TopSpin}).

\begin{figure}[h!]
\captionsetup[subfigure]{labelformat=simple}
\centering
\begin{tikzpicture}[scale=0.6]
    \draw[thick,fill,lightgray] (-3.25,0) circle (3);
    \draw[thick,fill,lightgray] (3.25,0) circle (3);
    \draw[thick,fill,lightgray] (-3.5,-3) -- (3.5,-3) -- (3.5,3) -- (-3.5,3) -- (-3.5,-3);
    \draw[thick,fill,lightgray] (0,2) circle (2.9);
    \draw[fill=plum] (0,2) circle (2.6);
    \draw (-0.1,1.35) -- (0.1,1.35) -- (0.1,-0.5) -- (-0.1,-0.5) -- (-0.1,1.35);
    \draw (-0.5,1.35) -- (-0.3,1.35) -- (-0.3,-0.45) -- (-0.5,-0.45) -- (-0.5,1.35);
    \draw (-0.9,1.35) -- (-0.7,1.35) -- (-0.7,-0.35) -- (-0.9,-0.35) -- (-0.9,1.35);
    \draw (-1.3,1.35) -- (-1.1,1.35) -- (-1.1,-0.15) -- (-1.3,-0.15) -- (-1.3,1.35);
    \draw (-1.7,1.35) -- (-1.5,1.35) -- (-1.5,0.15) -- (-1.7,0.15) -- (-1.7,1.35);
    \draw (-2.1,1.35) -- (-1.9,1.35) -- (-1.9,0.75) -- (-2.1,0.75) -- (-2.1,1.35);
    \draw (0.5,1.35) -- (0.3,1.35) -- (0.3,-0.45) -- (0.5,-0.45) -- (0.5,1.35);
    \draw (0.9,1.35) -- (0.7,1.35) -- (0.7,-0.35) -- (0.9,-0.35) -- (0.9,1.35);
    \draw (1.3,1.35) -- (1.1,1.35) -- (1.1,-0.15) -- (1.3,-0.15) -- (1.3,1.35);
    \draw (1.7,1.35) -- (1.5,1.35) -- (1.5,0.15) -- (1.7,0.15) -- (1.7,1.35);
    \draw (2.1,1.35) -- (1.9,1.35) -- (1.9,0.75) -- (2.1,0.75) -- (2.1,1.35);
    \draw (0,3.8) circle(.45);
    \draw[thick,fill,gray] (-3.25,2.7) arc (90:270:2.7) -- (-3.25,-1.45) arc (270:90:1.45) -- (-3.25,2.7);
    \draw[thick,fill,gray] (3.25,2.7) arc (90:-90:2.7) -- (3.25,-1.45) arc (-90:90:1.45) -- (3.25,2.7);
    \draw[thick,fill,gray] (-3.5,2.7) -- (3.5,2.7) -- (3.5,1.45) -- (-3.5,1.45) -- (-3.5,2.7);
    \draw[thick,fill,gray] (-3.5,-2.7) -- (3.5,-2.7) -- (3.5,-1.45) -- (-3.5,-1.45) -- (-3.5,-2.7);
    \draw[thick,fill=yellow] (-1.95,2.1) circle (0.6) node {\Large \bf 1};
    \draw[thick,fill=yellow] (-0.65,2.1) circle (0.6) node {\Large \bf 2};
    \draw[thick,fill=yellow] (0.65,2.1) circle (0.6) node {\Large \bf 3};
    \draw[thick,fill=yellow] (1.95,2.1) circle (0.6) node {\Large \bf 4};
    \draw[thick,fill=yellow] (3.25,2.1) circle (0.6) node {\Large \bf 5};
    \draw[thick,fill=yellow] (3.25,0) ++(54:2.1) circle (0.6) node {\Large \bf 6};
    \draw[thick,fill=yellow] (3.25,0) ++(18:2.1) circle (0.6) node {\Large \bf 7};
    \draw[thick,fill=yellow] (3.25,0) ++(-18:2.1) circle (0.6) node {\Large \bf 8};
    \draw[thick,fill=yellow] (3.25,0) ++(-54:2.1) circle (0.6) node {\Large \bf 9};
    \draw[thick,fill=yellow] (3.25,-2.1) circle (0.6) node {\Large \bf 10};
    \draw[thick,fill=yellow] (1.95,-2.1) circle (0.6) node {\Large \bf 11};
    \draw[thick,fill=yellow] (0.65,-2.1) circle (0.6) node {\Large \bf 12};
    \draw[thick,fill=yellow] (-0.65,-2.1) circle (0.6) node {\Large \bf 13};
    \draw[thick,fill=yellow] (-1.95,-2.1) circle (0.6) node {\Large \bf 14};
    \draw[thick,fill=yellow] (-3.25,-2.1) circle (0.6) node {\Large \bf 15};
    \draw[thick,fill=yellow] (-3.25,0) ++(234:2.1) circle (0.6) node {\Large \bf 16};
    \draw[thick,fill=yellow] (-3.25,0) ++(198:2.1) circle (0.6) node {\Large \bf 17};
    \draw[thick,fill=yellow] (-3.25,0) ++(162:2.1) circle (0.6) node {\Large \bf 18};
    \draw[thick,fill=yellow] (-3.25,0) ++(126:2.1) circle (0.6) node {\Large \bf 19};
    \draw[thick,fill=yellow] (-3.25,2.1) circle (0.6) node {\Large \bf 20};
\end{tikzpicture}
\caption{The (solved) Top Spin puzzle.}
\label{fig:TopSpin}
\end{figure}
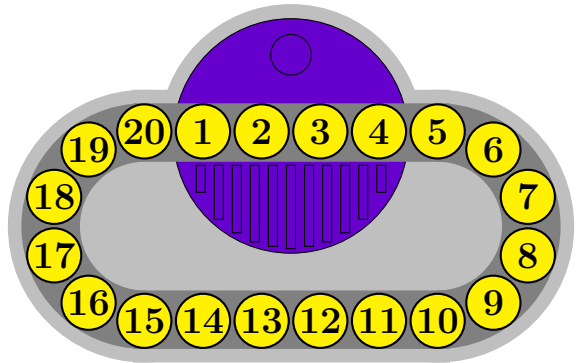

You can push the tiles around the track and the turntable allows you to flip 4 tiles at a time.  Starting with the numbers mixed up, your goal is to put them back in increasing order as you move clockwise around the track.  If we fix the left side of the turntable as the first position, then puzzle arrangements naturally correspond to permutations of the numbers $\{1, 2, \dots, 20\}$.  Thus, we can think of puzzle arrangements as sitting inside of the symmetric group $S_{20}$.

Of course, we can generalize the puzzle as well by considering the set of tiles $\{1, \dots, n\}$ for any natural number $n \in \N$ and by allowing for a turntable of size $k$ for any natural number $1 \leq k \leq n$.  These more general puzzles are often referred to as oval track puzzles and, just as in the Top Spin case, we can view puzzle arrangements as elements of the symmetric group $S_n$.

Now, imagine that you own one of these oval track puzzles and that you loan it out to a friend one day.  While attempting to solve the puzzle your friend becomes frustrated and smashes the puzzle on the table, thereby knocking out all of the tiles.  For a brief moment you consider yelling (to get your own frustration out), but your innate curiosity gets the better of you.  Your friend has inadvertently brought a fascinating conundrum to your attention.  You'd like to put your puzzle back together, but you want it to still be solvable (otherwise it will forever be a source of \emph{only} frustration).  Does it matter how you put the tiles back in?

As it turns out, if the puzzle you own is of the standard Top Spin variety, then the answer is no.  No matter how you put the tiles back in, it will always be solvable.  However, for some of the more general puzzles, it does depend on how you replace the tiles.  In the rest of the paper we answer this question in general by describing how the tiles must be replaced in order to ``fix'' a broken puzzle in each case.  Solving the puzzle corresponds to obtaining the identity permutation in $S_n$.  Thus, we determine which puzzle arrangements are solvable by considering the two moves in the oval track puzzle -- the \emph{translation} of the $n$ tiles around the track and the \emph{flip} of $k$ tiles using the turntable -- as generators of a subgroup of $S_n$.  For $n,k \in \N$ with $1 \leq k \leq n$, we will refer to the subgroup obtained as the \emph{oval track group}, $OT_{n,k}$.

\section{Preliminaries.}\label{sec:Prelims}
We start with a couple of important results known about the symmetric group which we plan to use, but we highly encourage any reader who is unfamiliar with the symmetric group to learn more about this beautiful group (see nearly any undergraduate text on abstract algebra, e.g.\ \cite{AH} or \cite{Nash}).  The first result is a well-known (see e.g.\ \cite[Lemma 2.7]{Rotman}) and extremely useful fact which describes what happens when you conjugate by permutations in the symmetric group.
\begin{lemma}\label{conjugate} Given $\theta, \sigma \in S_n$, then $\theta \sigma \theta^{-1}$ is the permutation obtained from $\sigma$ by applying $\theta$ to every number that appears in the disjoint cycle representation of $\sigma$.
\end{lemma}
 We will make use of conjugation several times throughout the paper so we give a quick example of Lemma~\ref{conjugate} in action.  Let $\theta = (1~2~3~4)$ and $\sigma=(1~3~5)(2~6)$ be permutations in $S_6$, then  
$$\theta \sigma \theta^{-1} = (\theta(1)~\theta(3)~\theta(5))(\theta(2)~\theta(6)) = (2~4~5)(3~6).$$

Since we aim to describe the group that our two basic moves generate, it will be helpful to know quickly when we generate the alternating group in various contexts.  More specifically, we will use the fact that (for $n \geq 3$) the elements of $A_n$ can be generated by a single consecutive 3-cycle together with the $n$-cycle $(1~2~\dots~n)$.

\begin{prop}\label{gensets}
 If $n \geq 3$ and $(i~i+1~i+2)$ is any consecutive 3-cycle, then $A_n$ is a subgroup of the group generated by $(i~i+1~i+2)$ and the $n$-cycle $(1~ 2~ \dots~ n)$.
\end{prop}
\begin{proof}
It is well-known (see e.g.\ \cite[Lemma 2]{Archer}) that the set of consecutive 3-cycles will generate all of $A_n$.  The result therefore follows as we may generate all of the consecutive 3-cycles using conjugation by the $n$-cycle.
\end{proof}


\subsection{The parity subgroup.}\label{subsec:Parity}
Observe that for any $n \in \N$, we can consider permutations in $S_n$ which maintain parity, i.e.\ ones which permute the odd numbers amongst themselves and simultaneously permute the even numbers amongst themselves.  We'll refer to these as \emph{Type I} permutations.  For example, our permutation $\sigma=(1~3~5)(2~6)$ from earlier is Type I.  If we isolate the odd portion of a Type I element, then we may think of it as a permutation in the symmetric group on the first $\lceil \frac{n}{2} \rceil$ odd numbers, which we denote by $S^{^\text{odd}}_{\lceil \frac{n}{2} \rceil}$.  Similarly, the even portion will be a permutation of the first $\lfloor \frac{n}{2}\rfloor$ even numbers, $S^{^\text{even}}_{\lfloor \frac{n}{2} \rfloor}$.  Thus, the Type I elements in $S_n$ exactly correspond to the elements in $S^{^\text{odd}}_{\lceil \frac{n}{2} \rceil} \times S^{^\text{even}}_{\lfloor \frac{n}{2} \rfloor}$ and it follows that there are $\left(\lceil \frac{n}{2} \rceil\right)!\left(\lfloor \frac{n}{2} \rfloor\right)!$ of these permutations.

In the special case when $n$ is even, it is also possible to have permutations which exactly swap the parity, i.e.\ ones which send all odd numbers to even numbers and vice versa.  We'll refer to these as \emph{Type II} permutations.  For example, the permutation $\theta=(1~2~3~4)$ is a Type II permutation in $S_4$, however it is \textbf{not} Type II in $S_6$ since in that group $\theta$ fixes 5 and 6 rather than changing their parity.  Observe that  there are exactly $\left(\frac{n}{2}\right)!\left(\frac{n}{2}\right)!$ different Type II permutations in $S_n$ as well.

When taken together, the union of Type I and Type II permutations, denoted by $PS_n$, forms what we call the \emph{parity subgroup} of $S_n$.  Thus, for any even $n \in \N$, we have $PS_n = \{\theta \in S_n \mid \theta \text{~is Type I or Type II}\}$.  Our observations above imply that $|PS_n| = 2\left(\frac{n}{2}\right)!\left(\frac{n}{2}\right)!$.


\subsection{Getting familiar with the oval track group.}\label{subsec:OvalTrack}
While we will make use of tools and notation from the symmetric group, it will also be helpful to think about the oval track group from the perspective of puzzle arrangements.  From that point of view we think of applying group elements as moves on the puzzle.  Towards that end, we will represent  arrangements as diagrams.  For example, the solved arrangement (or $(1) \in S_n$) is
$$\left(\circled{1} \dots \circled{k}\right) \circled{k+1} \dots \circled{n},$$
where the parentheses mark the location of the turntable.

Recall that the two basic moves that we have at our disposal are the \emph{translation}, which we will denote by $\tau$, and the \emph{flip}, which we will denote by $\phi$.  The translation rotates all tiles one position clockwise around the track, hence diagramatically:
$$\tau \left(\circled{1} \dots \circled{k}\right) \circled{k+1} \dots \circled{n} = \left(\circled{n}~\circled{1} \dots \circled{k-1}\right) \circled{k} \dots \circled{n-1}.$$
Observe that $n$ has moved to the first position, while each number $1 \leq i < n$ has moved to position $i+1$, thus, in cycle notation $\tau = (1~2~\dots~n)$.  Recall that since $\tau$ is an $n$-cycle, it can be written as a product of $n-1$ transpositions.  Hence, $\tau \in A_n$ when $n$ is odd, and $\tau \not \in A_n$ when $n$ is even.  Notice also that with $n$ tiles, if we continually translate in the same direction $n$ times we should be back to our starting position.  Hence $\tau^n=(1)$.

Meanwhile, the flip serves to reverse all of the numbers in the turntable, hence
$$\phi \left(\circled{1} \dots \circled{k}\right) \circled{k+1} \dots \circled{n} = \left(\circled{k} \dots \circled{1}\right) \circled{k+1} \dots \circled{n}.$$
Observe that this simultaneously swaps $\lfloor \frac{k}{2} \rfloor$ pairs of tiles, 1 and $k$, 2 and $k-1$, 3 and $k-2$ and so on.  Thus, when $k$ is even, we have $\phi=(1~k)(2~k-1)\dots(\frac{k}{2}~\frac{k}{2}+1)$ and when $k$ is odd, we have $\phi=(1~k)(2~k-1)\dots(\frac{k-1}{2}~\frac{k+3}{2})$, as $\frac{k+1}{2}$ will be fixed in the middle of the turntable.  Moreover, notice that the number of transpositions, $\lfloor \frac{k}{2} \rfloor$, is even -- meaning $\phi \in A_n$ -- when $k \equiv 0,1 \md$ and is odd -- meaning $\phi \not \in A_n$ -- when $k \equiv 2,3 \md$.  Certainly, if we were to immediately flip the turntable again we would be back to where we started.  Hence $\phi^2=(1)$.

With our two basic moves defined, we define the oval track group $OT_{n,k}$ for any $n \in \N$ and any integer $1 \leq k \leq n$ as the subgroup of $S_n$ generated by $\tau$ and $\phi$.  Hence, $OT_{n,k} = \langle \tau, \phi \rangle \leq S_n$ is the group we would like to describe in general.

\section{Degenerate cases.}
To start, we deal with some degenerate boundary cases.  We think of these as cases where either there is something broken about the puzzle or there isn't enough room for the full structure of $OT_{n,k}$ to be realized.  

\begin{prop}[Degenerate Cases]\label{degenerate}~\\ \vspace{-0.5cm}
\begin{itemize}
\item[(i)] $OT_{n,1} \cong \Z_n$,
\item[(ii)] $OT_{2,2} \cong S_2$,
\item[(iii)] For $n \geq 3$, $OT_{n,n} \cong OT_{n,n-1} \cong D_n$, the symmetries of a regular $n$-gon.
\end{itemize}
\end{prop}
\begin{proof}
For (i), we have $k=1$ and thus, our turntable does not actually accomplish anything, i.e.\ $\phi=(1)$.  Because of this, we are left with a single generator $\tau$ of order $n$.  For (ii), $\phi = \tau = (1~2)$ which generates $S_2$.  For (iii), we will show that the oval track group has the same presentation as $D_n$ in terms of generators and relations.  
Recall that the dihedral group $D_n$ can be presented as the group generated by a reflection $h$ and a rotation $t$ subject to the relations $h^2=\Id$, $t^n=\Id$, and $hth=t^{-1}$.  We have already seen for the oval track group that $\phi^2=(1)$ and $\tau^n=(1)$, so we need only to show that $\phi \tau \phi = \tau^{-1}$ to complete the proof.  Since $\phi=\phi^{-1}$, the left side is conjugation of $\tau$ by $\phi$.  Thus, by Lemma~\ref{conjugate}, when $k=n$ we have
$$\phi \tau \phi^{-1} = (\phi(1)~\phi(2)~\dots~\phi(n)) = (n~n-1~\dots~1) = \tau^{-1},$$
and similarly, when $k=n-1$ we have
$$\phi \tau \phi^{-1} = (\phi(1)~\phi(2)~\dots~\phi(n-1)~\phi(n)) = (n-1~n-2~\dots~1~n) = \tau^{-1},$$
since $\phi$ fixes $n$ which is outside the turntable.
\end{proof}

In terms of fixing broken puzzles, this means that if you have a puzzle with $k=1$, it can only be fixed by returning the tiles in consecutive order clockwise around the track with any tile in first position (or, equivalently, starting with 1 in any position).  In the case when $n=k=2$, the tiles can be replaced in either of the two possible ways.  And when $n \geq 3$ and $k=n$ or $n-1$, we can imagine using the tiles as labels for the vertices of a regular $n$-gon.  Just as every symmetry in $D_n$ can be completely described by giving the location of a particular label and an orientation of the labels, so too can we describe the possible ways of fixing the puzzle in this fashion.  We can place any tile we choose into any position after which they must continue consecutively either clockwise or counterclockwise around the track.

\section{Nice moves, permutations, and subgroups.}
Observe that our degenerate cases serve to cover all cases with $1 \leq n \leq 3$ and also any case with $k=1$, or with $k=n$ or $n-1$.  Thus, as we move to the more general situation, we will assume for the rest of the paper that $n \geq 4$ and that $1<k<n-1$.  We would also like to note that some of the techniques that follow are due to Kaufmann and Kavountzis (see \cite{KK} and \cite{Kaufmann}), however we found several errors in their work and thus we have chosen to reorganize and present our own arguments where our work overlaps theirs.

To give a general outline of what follows: First we will define some helpful puzzle moves which we plan to use heavily.  Second we will use those moves to generate useful permutations, such as consecutive 3-cycles.  And third we will use those permutations together with the structure of $\phi$ and $\tau$ to find ``subgroup bounds'' on $OT_{n,k}$ inside $S_n$.  As we saw earlier
, features of our generators depend on the parity of $n$ and on $k$ modulo 4, thus the structures of our nice moves and the work that follows will require us to split into various subcases.

\subsection{Flip-translations and shuffles.}
We will call $\rho = \phi \tau$ a \emph{flip-translation}.  The purpose of a flip-translation is essentially to move tiles from one side of the turntable to the other.  More specifically, if we imagine putting the left-most $k-1$ tiles in the turntable into a block together, then $\rho$ has the effect of moving the first tile to the left of the turntable to the right-side of the turntable while simultaneously reversing the order of the tiles in our block:
$$\rho ~~ \left(\block{\circled{1} \dots \circled{k-1}}~ \circled{k}\right) \circled{k+1} \dots \circled{n} = \left(\block{\circled{k-1} \dots \circled{1}}~ \circled{n}\right) \circled{k} \dots \circled{n-1}$$
Similarly, $\rho^{-1}$ allows us to move tiles in the other direction
and certainly if we apply an even number of flip-translations (or inverse flip-translations), then the block will be back into its usual (clockwise) order.

Another interesting and useful move is $\pi=\tau\phi\tau^{-1}\phi = \tau^{2}\rho^{-2}$ which we will refer to as a \emph{shuffle}.  Observe in general (recall $n\geq k+2$) that $\pi$ acts as follows:
$$
\pi \left(\circled{1} \dots \circled{k}\right) \circled{k+1} \dots \circled{n} = \left(\circled{k}~\circled{k+1}~\circled{1} \dots \circled{k-2}\right) \circled{k-1}~ \circled{k+2} \dots \circled{n}.
$$
Hence, when $k$ is even, this element is the $(k+1)$-cycle $(1~3~5~\dots~k+1~2~4~\dots~k)$ that cycles through the odds first and then the evens.  When $k$ is odd instead, this element is the pair of disjoint $\frac{k+1}{2}$-cycles $(1~3~5~\dots~k)(2~4~\dots~k+1)$ where one cycles through the odds and the other cycles through the evens separately.  

From the definition of $\pi$, we can see that $\pi^{-1}$ will take the two tiles at the left of the turntable and will insert them after the first tile to the right of the turntable, leaving the 3rd tile in the turntable as the new first tile (see below):
$$\pi^{-1}\left(\circled{1} \dots \circled{k}\right) \circled{k+1} \dots \circled{n} = \left(\circled{3} \dots \circled{k+1}~\circled{1}\right) \circled{2}~\circled{k+2} \dots \circled{n}$$

\begin{lemma}\label{k-cycles} (a) If $n-k$ is even, then $\tau\rho^{n-k}$ is the consecutive $k$-cycle $(1~2~\dots~k)$.\\
(b) If $k$ is even, $\pi^{\frac{k}{2}}$ is the consecutive $(k+1)$-cycle $(k+1~k~\dots~1)$.
\end{lemma}
\begin{proof}
(a) As discussed above, the key observation is that each flip-translation moves one tile from the left of the turntable to the right.  Hence, after performing this procedure $n-k$ times, we will have moved each of the tiles $n,~ n-1,~ \dots,~ k+1$, one at a time, from the left side of the window to the right side.  Moreover, since $n-k$ is even, our block will be in the proper order -- see below:
$$
\begin{tabular}{c} $\tau\rho^{n-k} \left(\block{\circled{1} \dots \circled{k-1}}~ \circled{k}\right) \circled{k+1} \dots \circled{n}$\\
$=\tau\left(\block{\circled{1} \dots \circled{k-1}}~\circled{k+1}\right) \circled {k+2} \dots \circled{n}~ \circled{k}$\\
$=\left(\circled{k} ~ \block{\circled{1} \dots \circled{k-1}}\right) \circled {k+1} \dots \circled{n}$
\end{tabular}
$$
Hence, $\tau \rho^{n-k} = (1~2~\dots~k)$ as claimed.\\
(b) Observe that when $k$ is even, using $\pi$ (which moves two tiles at a time) $\frac{k}{2}$ times serves to exactly move the $k$ tiles $2, 3, \dots, k+1$ in order to the left of 1, while leaving all other tiles fixed.  Hence
$$\pi^{k/2} \left(\circled{1} \dots \circled{k}\right) \circled{k+1} \dots \circled{n} = \left(\circled{2} \dots \circled{k+1}\right) \circled{1}~\circled{k+2} \dots \circled{n}$$
and $\pi^{k/2} = (k+1~k~\dots~1)$ as claimed.
\end{proof}

\subsection{Creating 3-cycles.}
With flip-translations, shuffles, and some $k$ and $(k+1)$-cycles at our disposal, we now create 3-cycles of various varieties.  Recall that consecutive 3-cycles will allow us to generate $A_n$ by Proposition~\ref{gensets}.

\begin{lemma}\label{3-cycles} (a) If $n-k$ is even, $\pi^{-1}\left(\tau \rho^{n-k}\right)^2$ is the 3-cycle $(k-1~k~k+1)$.\\
(b) If $k$ is even, $\pi^{k/2} \tau \pi^{k/2} \phi \rho^{-1}$ is the 3-cycle $(k~k+1~k+2)$.
\end{lemma}
\begin{proof} (a) We use the inverse shuffle, $\pi^{-1}$, and the $k$-cycle $\tau\rho^{n-k}$ from Lemma~\ref{k-cycles} since $n-k$ is even.
$$
\begin{tabular}{c} $\pi^{-1}\left(\tau\rho^{n-k}\right)^2 \left(\circled{1} \dots \circled{k}\right) \circled{k+1} \dots \circled{n}$\\
$\pi^{-1}\left(\tau\rho^{n-k}\right) \left(\circled{k}~\circled{1} \dots \circled{k-1}\right) \circled{k+1} \dots \circled{n}$\\
$\pi^{-1} \left(\circled{k-1}~\circled{k}~\circled{1} \dots \circled{k-2}\right) \circled{k+1} \dots \circled{n}$\\
$\left(\circled{1} \dots \circled{k-2}~\circled{k+1}~\circled{k-1} \right)\circled{k}~ \circled{k+2} \dots \circled{n}$.
\end{tabular}
$$
Observe that all tiles are in their original positions, except $k-1$ (now in position $k$), $k$ (now in position $k+1$), and $k+1$ (now in position $k-1$).\\
(b) Here we use the $(k+1)$-cycle $\pi^{k/2}$ from Lemma~\ref{k-cycles} since $k$ is even.
$$
\begin{tabular}{c} $\pi^{k/2} \tau \pi^{k/2} \phi \rho^{-1}\left(\circled{1} \dots \circled{k}\right) \circled{k+1} \dots \circled{n}$\\
$\pi^{k/2} \tau \pi^{k/2} \phi \left(\circled{k-1} \dots \circled{1}~\circled{k+1}\right) \circled{k+2} \dots \circled{n}~\circled{k}$\\
$\pi^{k/2} \tau \pi^{k/2} \left(\circled{k+1}~ \circled{1} \dots \circled{k-1}\right) \circled{k+2} \dots \circled{n} ~ \circled{k} $\\
$\pi^{k/2} \tau \left(\circled{1} \dots \circled{k-1}~\circled{k+2}\right) \circled{k+1}~\circled{k+3} \dots \circled{n} ~ \circled{k}$\\
$\pi^{k/2} \left(\circled{k}~\circled{1} \dots \circled{k-1}\right) \circled{k+2}~\circled{k+1}~\circled{k+3} \dots \circled{n}$\\
$\left(\circled{1} \dots \circled{k-1}~\circled{k+2}\right)\circled{k}~\circled{k+1}~\circled{k+3} \dots \circled{n}$.
\end{tabular}
$$
Observe that all tiles are in their original positions, except $k$ (now in position $k+1$), $k+1$ (now in position $k+2$), and $k+2$ (now in position $k$).
\end{proof}

\begin{lemma}\label{kodd3-cycle} If $k$ is odd, the element $\pi\tau\pi^{-1}\tau^{-1}$ is the 3-cycle $(1~3~k+2)$.
\end{lemma}
\begin{proof} The key idea here is that, $\tau\pi\tau^{-1}$ is equivalent to applying $\pi$ one position to the right, hence $\tau \pi \tau^{-1}$ is the pair of disjoint $\frac{k+1}{2}$-cycles $(2~4~\dots~k+1)(3~5~\dots~k~k+2)$.  The even portion of $\tau \pi \tau^{-1}$ is exactly the same as the even portion from $\pi$ itself, while the odd part is different.  Thus, if we multiply $\pi$ together with $\left(\tau \pi \tau^{-1}\right)^{-1}$, we will obtain a permutation that only permutes odd tiles:
$$
\begin{tabular}{c} $\pi\tau\pi^{-1}\tau^{-1} \left(\circled{1} \dots \circled{k}\right) \circled{k+1} \dots \circled{n}$\\
$\pi\tau\pi^{-1} \left(\circled{2} \dots \circled{k+1}\right) \circled{k+2} \dots \circled{n}~\circled{1}$\\
$\pi\tau \left(\circled{4} \dots \circled{k+1}~\circled{k+2}~\circled{2}\right) \circled{3}~\circled{k+3} \dots \circled{n}~\circled{1}$\\
$\pi \left(\circled{1}~\circled{4} \dots \circled{k+1}~\circled{k+2}\right)\circled{2}~ \circled{3}~\circled{k+3} \dots \circled{n}$\\
$\left(\circled{k+2}~\circled{2}~\circled{1}~\circled{4} \dots \circled{k}\right)\circled{k+1}~ \circled{3}~\circled{k+3} \dots \circled{n}$
\end{tabular}
$$
Notice that $1$ is in the 3rd position, $3$ is in the $(k+2)$-th position, and $k+2$ is in the 1st position, while all other tiles are in their original positions.
\end{proof}

\begin{prop}\label{consecutiveodds} If $n$ is even and $k$ is odd, then the consecutive odd 3-cycle $(1~3~5)$ is in the oval track group $OT_{n,k}$.
\end{prop}
\begin{proof} Here it makes much more sense to use symmetric group tools and notation rather than dealing with puzzles.  First we'll look at the case when $n \geq k+4$ (which really means $n \geq k+5$ since $n$ is even).  In this case, we let $\Gamma=\tau^3 \pi \tau^{-3}$ and observe that $\Gamma$ is the pair of $\frac{k+1}{2}$-cycles $(4~6~\dots~k+3)(5~7~\dots~k+2~k+4)$ by Lemma~\ref{conjugate}.  Hence, conjugating the 3-cycle $\pi\tau\pi^{-1}\tau^{-1}=(1~3~k+2)$ from Lemma~\ref{kodd3-cycle} by $\Gamma$ gives:
$$\Gamma^2 (\pi \tau \pi^{-1} \tau^{-1})\Gamma^{-2} = (\Gamma^2(1)~\Gamma^2(3)~\Gamma^2(k+2)) = (1~3~5)$$
Unfortunately, this method only works when $n \geq k+4$, hence we must deal with $n=k+3$ separately (notice $n=k+2$ makes no sense since $n$ is even and $k$ is odd).  Here, our 3-cycle is $\pi \tau \pi^{-1} \tau^{-1} = (1~3~n-1)$.  Conjugating by $\tau^2$ we obtain:
$$\tau^2 (\pi \tau \pi^{-1} \tau^{-1}) \tau^{-2} = \tau^2 (1~3~n-1) \tau^{-2} = (\tau^2(1)~\tau^2(3)~\tau^2(n-1)) = (3~5~1).$$  
Thus, in all cases we have generated the consecutive odd 3-cycle $(1~3~5)$.
\end{proof}

\subsection{Subgroup bounds.}
With all of the nice permutations we've constructed, we will now spend time putting some helpful bounds on $OT_{n,k}$ by both generating subgroups of $OT_{n,k}$ and also finding proper subgroups of $S_n$ which $OT_{n,k}$ must live inside in the various cases.  These bounding groups will make it easier to describe the group $OT_{n,k}$ in the general setting.

\begin{lemma}\label{nicecases} If $n$ and $k$ are both even, or if $n$ is odd, then $A_n \leq OT_{n,k}$.
\end{lemma}
\begin{proof} Since we always have the $n$-cycle $\tau=(1~2~\dots~n)$, it suffices to show that we can construct a consecutive 3-cycle in each case by Proposition~\ref{gensets}.  By Lemma~\ref{3-cycles}, if $n$ and $k$ have the same parity, we can generate $(k-1~k~k+1)$ and if $n$ is odd and $k \equiv 0,2 \md$, then we have $(k~k+1~k+2)$.
\end{proof}

Lemma~\ref{nicecases} gives us a strong lower bound in these cases as now $OT_{n,k}$ must be either $A_n$ or $S_n$ by Lagrange's Theorem.  This still leaves us with the situation where $n$ is even and $k \equiv 1,3 \md$.  These cases are distinctly different from the others as $\tau$ is a Type II element when $n$ is even, and $\phi$ is a Type I element when $k \equiv 1,3 \md$
.  Thus, as our two generators are elements of the parity subgroup, $PS_n$, it follows that $OT_{n,k} \leq PS_n$ in these cases.  As it turns out, the two potential cases here must also be separated, however they do share some structure.

First observe that if we have any permutation $\alpha \in OT_{n,k}$ that is Type II, then there exists a Type I permutation $\beta =\tau^{-1}\alpha \in OT_{n,k}$ such that $\alpha=\tau \beta$.  In view of this observation, we can reduce the problem of describing the group $OT_{n,k}$ to one of describing only the Type I permutations, which are contained in $S^{^{\text{odd}}}_{\frac{n}{2}} \times S^{^{\text{even}}}_{\frac{n}{2}}$, instead.

\begin{prop}\label{weirdnormal} (a) If $n$ is even and $k \equiv 1,3 \md$, then $A^{^{\text{odd}}}_{\frac{n}{2}} \times A^{^{\text{even}}}_{\frac{n}{2}} \leq OT_{n,k}$.\\
(b) If $n$ is even, and $k \equiv 3 \md$, then $S^{^{\text{odd}}}_{\frac{n}{2}} \times S^{^{\text{even}}}_{\frac{n}{2}} \leq OT_{n,k}$.
\end{prop}
\begin{proof} (a) Proposition~\ref{gensets} applied to the alternating group on only odd numbers (or even numbers) tells us that $A^{^\text{odd}}_{\frac{n}{2}}$ is generated by the consecutive \emph{odd} 3-cycles $(1~3~5), (3~5~7), \dots, (n-5~n-3~n-1)$ and similarly, $A^{^\text{even}}_{\frac{n}{2}}$ is generated by the consecutive even 3-cycles.  By Proposition~\ref{consecutiveodds}, we can generate the consecutive odd 3-cycle $(1~3~5)$.  Thus, with repeated conjugation by $\tau$ we obtain the complete set of both the consecutive odd 3-cycles and the consecutive even 3-cycles.\\
(b) Recall that when $k \equiv 3 \md$, then $\phi$ will have an odd number of transpositions.   Observe that there must be an even number of transpositions involving only odd numbers or involving only even numbers.  That portion of $\phi$ will be an element in $A^{^{\text{odd}}}_{\frac{n}{2}}$ or in $A^{^{\text{even}}}_{\frac{n}{2}}$ and thus its inverse (itself) is in $OT_{n,k}$ by part (a).  Multiplying that inverse by $\phi$ will leave us with an odd element in either $S^{^{\text{odd}}}_{\frac{n}{2}}$ or $S^{^{\text{even}}}_{\frac{n}{2}}$.  We can then conjugate by $\tau$ to obtain an odd element of the other variety.  It follows that we can generate all of $S^{^{\text{odd}}}_{\frac{n}{2}} \times S^{^{\text{even}}}_{\frac{n}{2}}$ by Lagrange's Theorem.
\end{proof}

When $n$ is even and $k\equiv 3 \md$ we have just seen that $OT_{n,k}$ contains all of the Type I permutations in $S_n$.  Recall, however, that when $k \equiv 1 \md$ instead, then $\phi \in A_n$.  Thus, if $\alpha \in OT_{n,k} = \langle \phi, \tau \rangle$ is a Type I permutation then it follows that $\alpha \in S^{^{\text{odd}}}_{\frac{n}{2}} \times S^{^{\text{even}}}_{\frac{n}{2}} \cap A_n$ as any way of writing it as a product of $\phi$'s and $\tau$'s must have an even number of $\tau$'s.  Hence, $OT_{n,k}$ definitely does not contain the full set of Type I permutations when $k \equiv 1 \md$.  To describe the Type I permutations that are generated we must split this case even further.

Observe that when $k\equiv 1 ~(\operatorname{mod}~8)$, $\lfloor\frac{k}{2}\rfloor$ is a number divisible by 4, hence $\phi$ will be in $A^{^{\text{odd}}}_{\frac{n}{2}} \times A^{^{\text{even}}}_{\frac{n}{2}}$.  However, if $k\equiv 5 ~(\operatorname{mod}~8)$ instead, then $\lfloor \frac{k}{2} \rfloor$ is divisible by 2, but \textbf{not} by 4.  Hence $\phi$ will be in $\left(S^{^{\text{odd}}}_{\frac{n}{2}} \times S^{^{\text{even}}}_{\frac{n}{2}} \cap A_n\right) \setminus A^{^{\text{odd}}}_{\frac{n}{2}} \times A^{^{\text{even}}}_{\frac{n}{2}}$.  Moreover, since $n$ is even, $\tau^2 = (1~3~\dots~n-1)(2~4~\dots~n) \in A_n$ and is Type I.  Each cycle is an $\frac{n}{2}$-cycle, hence $\tau^2 \in A^{^{\text{odd}}}_{\frac{n}{2}} \times A^{^{\text{even}}}_{\frac{n}{2}}$ if and only if $n \equiv 2 \md$.

\begin{prop}\label{typeIeven} If $n \equiv 0 \md$ OR $k\equiv 5 ~(\operatorname{mod}~8)$, then $S^{^{\text{odd}}}_{\frac{n}{2}} \times S^{^{\text{even}}}_{\frac{n}{2}} \cap A_n$ is a subgroup of $OT_{n,k}$.
\end{prop}
\begin{proof} Observe that $A^{^{\text{odd}}}_{\frac{n}{2}} \times A^{^{\text{even}}}_{\frac{n}{2}}$ makes up exactly half of $S^{^{\text{odd}}}_{\frac{n}{2}} \times S^{^{\text{even}}}_{\frac{n}{2}} \cap A_n$ and that we already know $A^{^{\text{odd}}}_{\frac{n}{2}} \times A^{^{\text{even}}}_{\frac{n}{2}} \leq OT_{n,k}$ by Proposition~\ref{weirdnormal}.  Thus, by Lagrange's Theorem, it suffices to show that we can generate one element in $S^{^{\text{odd}}}_{\frac{n}{2}} \times S^{^{\text{even}}}_{\frac{n}{2}} \cap A_n$ where both parts are odd permutations.  Observe that when $n \equiv 0 \md$, 
$\tau^2$ is a permutation in $S^{^{\text{odd}}}_{\frac{n}{2}} \times S^{^{\text{even}}}_{\frac{n}{2}} \cap A_n$ where both parts are odd permutations.  Similarly, when $k\equiv 5 ~(\operatorname{mod}~8)$, $\phi$ is such a permutation as well.
\end{proof}

\begin{prop}\label{ntwokone} If $n \equiv 2 \md$, $k\equiv 1 ~(\operatorname{mod}~8)$, and $\alpha \in OT_{n,k}$ is a Type I permutation, then $\alpha \in A^{^{\text{odd}}}_{\frac{n}{2}} \times A^{^{\text{even}}}_{\frac{n}{2}}$.
\end{prop}
\begin{proof} First, observe that if $\alpha$ is Type I, then $\tau \alpha \tau^{-1}$ will also be Type I.  Moreover,  we know that $\alpha$ can be written as a product of $\phi$'s and $\tau$'s.  We now prove the claim by induction on the number of terms in the product.  Since $\tau$ is Type II, the base case is when $\alpha =\phi$ which is in $A^{^{\text{odd}}}_{\frac{n}{2}} \times A^{^{\text{even}}}_{\frac{n}{2}}$.

Now, suppose that for any $\alpha=\alpha_1 \alpha_2 \cdots \alpha_t$ with $t \geq 1$ and each $\alpha_i = \phi$ or $\tau$ we have $\alpha \in A^{^{\text{odd}}}_{\frac{n}{2}} \times A^{^{\text{even}}}_{\frac{n}{2}}$.  Consider the case when $\alpha$ is Type I and has minimal presentation $\alpha_1 \cdots \alpha_{t+1}$.  If $\alpha_i = \tau$ for all $i$, then it follows that $t+1$ is even and that $\alpha \in A^{^{\text{odd}}}_{\frac{n}{2}} \times A^{^{\text{even}}}_{\frac{n}{2}}$ since $\tau^2 \in A^{^{\text{odd}}}_{\frac{n}{2}} \times A^{^{\text{even}}}_{\frac{n}{2}}$.  If not, then let $s$ be minimal such that $\alpha_{t+1-s}=\phi$.  Hence, $\alpha=\alpha_1\dots\alpha_{t-s} \phi \tau^{\pm s}$.  As these two cases are similar, we'll deal only with the case when the exponent on $\tau$ is positive for the sake of clarity.  We may conjugate $s$ times by $\tau$, to get $\phi$ to be the rightmost element, and then multiply by $\phi$, to obtain the Type I permutation
$$\gamma = \tau^{s} \alpha \tau^{-s} \phi = \tau^{s} \alpha_1 \cdots \alpha_{t-s}\phi \phi = \tau^{s} \alpha_1 \cdots \alpha_{t-s}.$$
This product has only $t$ terms, thus it follows by induction that $\gamma \in A^{^{\text{odd}}}_{\frac{n}{2}} \times A^{^{\text{even}}}_{\frac{n}{2}}$.  Now, as $\alpha = \tau^{-s} \left(\gamma \phi\right)\tau^{s}$ it follows that $\alpha \in A^{^{\text{odd}}}_{\frac{n}{2}} \times A^{^{\text{even}}}_{\frac{n}{2}}$, completing the induction.
\end{proof}

\section{Describing $OT_{n,k}$ and fixing puzzles.}
We now describe the oval track group $OT_{n,k}$ in each case and follow up with an interpretation of each group as a set of instructions to fix any broken puzzles.  Since the oval track group elements exactly correspond to the \emph{solvable} puzzle positions, we can fix any broken puzzle by replacing the tiles using appropriate permutations.

\begin{main}If $n \geq 4$ and $1 < k < n-1$, then
\begin{itemize}
\item[(1)] $OT_{n,k} \cong S_n$ if $n$ is even and $k$ is even OR if $n$ is odd and $k \equiv 2,3 \md$.
\item[(2)] $OT_{n,k} \cong A_n$ if $n$ is odd and $k \equiv 0,1 \md$.
\item[(3)] $OT_{n,k} \cong PS_n$ if $n$ is even and $k \equiv 3 \md$.
\item[(4)] $OT_{n,k} \cong \{\alpha, \tau\alpha \mid \alpha \in S^{^{\text{odd}}}_{\frac{n}{2}} \times S^{^{\text{even}}}_{\frac{n}{2}} \cap A_n\}$ if $n$ is even and $k \equiv 5 ~(\operatorname{mod}~8)$\\{\color{white}.}~~~ OR if $n \equiv 0 \md$ and $k \equiv 1 ~(\operatorname{mod}~8)$.
\item[(5)] $OT_{n,k} \cong \{\alpha, \tau\alpha \mid \alpha \in A^{^{\text{odd}}}_{\frac{n}{2}} \times A^{^{\text{even}}}_{\frac{n}{2}}\}$ if $n \equiv 2 \md$ and $k \equiv 1 ~(\operatorname{mod}~8)$.
\end{itemize}
\end{main}
\begin{proof} For (1) and (2), Lemma~\ref{nicecases} tells us that $A_n \leq OT_{n,k}$.  Recall that when $n$ is even, $\tau$ is an odd permutation and when $k \equiv 2,3 \md$, then $\phi$ is an odd permutation.  Hence, in each case of (1) we can generate all of $S_n$ by Lagrange's Theorem.  Meanwhile, for the cases in (2), both generators are contained in $A_n$ and thus it follows that $OT_{n,k} \cong A_n$.

For (3), (4), and (5) recall that we have already seen that $A^{^{\text{odd}}}_{\frac{n}{2}} \times A^{^{\text{even}}}_{\frac{n}{2}} \leq OT_{n,k} \leq PS_n$.  Moreover, we know that the Type II portion of $OT_{n,k}$ is exactly the coset of the Type I portion generated by $\tau$.  The results now follow by Propositions~\ref{weirdnormal},~\ref{typeIeven}, and ~\ref{ntwokone} respectively.
\end{proof}

Recall that the set of all potential puzzle arrangements exactly corresponds to the permutations in $S_n$.  Thus, for cases in collection (1), it does not matter how the tiles are returned when fixing the puzzle.  If instead, you are in collection (2), then you must be more careful.  One way to proceed here is to build cycles.  Start by placing any tile anywhere in the puzzle.  Then whatever location you filled, pick the corresponding tile up next (i.e.\ the tile that would inhabit that position in the solved puzzle) and place it anywhere.  Continue this procedure until you fill the location of a tile you've already picked up (which will necessarily be the location corresponding to the first tile you picked up) -- this creates one cycle.  Then pick up another tile you haven't placed yet and continue the process.

Since we are building a permutation out of cycles, we can determine whether the puzzle will be solvable by seeing if we can decompose those cycles into an even number of transpositions.  Recall that a cycle with $\ell$ tiles can be decomposed in terms of $\ell-1$ transpositions.  Thus, when decomposing all of the disjoint cycles created (which together cover all $n$ tiles) we will arrive at $n-c$ transpositions, where $c$ is the number of cycles.  Since $n$ is odd, this number will be even (and hence the puzzle will be fixed) if and only if $c$ is odd.

For collection (3), you can replace the tiles however you would like as long as tiles with the same parity never end up next to each other.  Another way to think about this is to color the odd positions blue and the even positions red.  Then you can replace the tiles in any way that places the odd tiles in one color and the even tiles in the other.

For collections (4) and (5), we must try to combine the instructions for (2) and (3).  We can create cycles again, but we need to be more careful in their construction.  The difference is that we also need to separate tiles of different parity.  One way to accomplish this is to color odd tiles blue and even tiles red and then renumber the tiles of each color from 1 to $\frac{n}{2}$ (in increasing order).  In addition, we must mentally separate the locations by parity (coloring them blue and red again) and then imagine renumbering each half of the track from 1 to $\frac{n}{2}$ as well.  We then assign a half of the track to each pile and create cycles as before, but from one pile at a time.  This gives us an alternating blue and red track and corresponding piles of blue and red tiles numbered from 1 to $\frac{n}{2}$ to work with.  If we swap colors (putting red tiles in blue locations and vice versa), then this corresponds to creating a Type II permutation.  Since Type II permutations have the form $\tau \alpha$ for some Type I $\alpha$, we account for this by shifting the position labels so that the first odd (blue) position is in position three.

Using that method of construction, we may satisfy the conditions of collection (4), by having an odd number of cycles in total.  One possible example in this collection would be a puzzle with $n=20$ and $k=5$.  Using blue and red tiles and positions each renumbered 1 through 10 instead of $1, 3, \dots, 19$ we then create cycles out of each color.  If we choose to place blue tiles in red positions and vice versa, then the blue positions labeled 1 through 10 correspond to the old positions 3, 5, $\dots$, 19, 1, respectively.  In blue, we could make the permutation $(5~6~7~10)(2~4~1~3)(8)(9)$ (which is four cycles), and in red we could make $(3~4)(2)(1~5)(6)(7~10~9)(8)$ (six cycles), for a total of ten cycles.  In translating back to evens and odds, recall that we also swapped all evens with all odds, thus obtaining $(1~6~9~12~13~20~19~10~3~8~7~2~11~14)(4~5)(15~16~17~18)$ which is an odd, Type II permutation (whose Type I part is in $A_n$) as desired.  See Figure~\ref{fig:collection4} for the construction and corresponding puzzle.

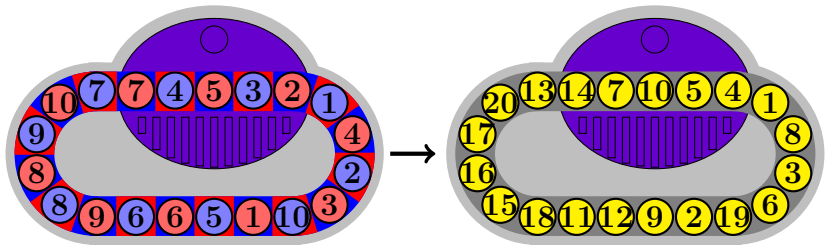
\begin{figure}[h!]
\captionsetup[subfigure]{labelformat=simple}
\centering
\begin{tikzpicture}[scale=0.4]
    \draw[thick,fill,lightgray] (-3.25,0) circle (3);
    \draw[thick,fill,lightgray] (3.25,0) circle (3);
    \draw[thick,fill,lightgray] (-3.5,-3) -- (3.5,-3) -- (3.5,3) -- (-3.5,3) -- (-3.5,-3);
    \draw[thick,fill,lightgray] (0.65,2) ellipse (3.6 and 2.9);
    \draw[fill=plum] (0.65,2) ellipse (3.2 and 2.5);
    \begin{scope}[shift={(0.65,0)},xscale=1.2,yscale=0.9]
    \draw (-0.1,1.35) -- (0.1,1.35) -- (0.1,-0.5) -- (-0.1,-0.5) -- (-0.1,1.35);
    \draw (-0.5,1.35) -- (-0.3,1.35) -- (-0.3,-0.45) -- (-0.5,-0.45) -- (-0.5,1.35);
    \draw (-0.9,1.35) -- (-0.7,1.35) -- (-0.7,-0.35) -- (-0.9,-0.35) -- (-0.9,1.35);
    \draw (-1.3,1.35) -- (-1.1,1.35) -- (-1.1,-0.15) -- (-1.3,-0.15) -- (-1.3,1.35);
    \draw (-1.7,1.35) -- (-1.5,1.35) -- (-1.5,0.15) -- (-1.7,0.15) -- (-1.7,1.35);
    \draw (-2.1,1.35) -- (-1.9,1.35) -- (-1.9,0.75) -- (-2.1,0.75) -- (-2.1,1.35);
    \draw (0.5,1.35) -- (0.3,1.35) -- (0.3,-0.45) -- (0.5,-0.45) -- (0.5,1.35);
    \draw (0.9,1.35) -- (0.7,1.35) -- (0.7,-0.35) -- (0.9,-0.35) -- (0.9,1.35);
    \draw (1.3,1.35) -- (1.1,1.35) -- (1.1,-0.15) -- (1.3,-0.15) -- (1.3,1.35);
    \draw (1.7,1.35) -- (1.5,1.35) -- (1.5,0.15) -- (1.7,0.15) -- (1.7,1.35);
    \draw (2.1,1.35) -- (1.9,1.35) -- (1.9,0.75) -- (2.1,0.75) -- (2.1,1.35);
    \end{scope}
    \draw (0.65,3.8) circle(.45);
    \draw[thick,fill,blue] (-3.25,2.7) arc (90:270:2.7) -- (-3.25,-1.45) arc (270:90:1.45) -- (-3.25,2.7);
    \draw[thick,fill,blue] (3.25,2.7) arc (90:-90:2.7) -- (3.25,-1.45) arc (-90:90:1.45) -- (3.25,2.7);
    \draw[thick,fill,blue] (-3.5,2.7) -- (3.5,2.7) -- (3.5,1.45) -- (-3.5,1.45) -- (-3.5,2.7);
    \draw[thick,fill,blue] (-3.5,-2.7) -- (3.5,-2.7) -- (3.5,-1.45) -- (-3.5,-1.45) -- (-3.5,-2.7);
    \draw[thick,fill,red] (-2.62,2.7) -- (-3.25,2.7) arc(90:107:2.7) -- ++(107:-1.25) arc(107:90:1.45) -- (-2.62,1.45) -- cycle;
    \draw[thick,fill,red] (-3.25,0) +(144:1.45) arc(144:180:1.45) -- +(180:1.25) arc (180:144:2.7);
    \draw[thick,fill,red] (-3.25,0) +(216:1.45) arc(216:251:1.45) -- +(251:1.25) arc (251:216:2.7);
    \draw[thick,fill,red] (-2.57,-2.7) -- (-1.33,-2.7) -- (-1.33,-1.45) -- (-2.57,-1.45) -- cycle;
    \draw[thick,fill,red] (-1.27,2.7) -- (-0.03,2.7) -- (-0.03,1.45) -- (-1.27,1.45) -- cycle;
    \draw[thick,fill,red] (0.03,-2.7) -- (0.03,-1.45) -- (1.27,-1.45) -- (1.27, -2.7) -- cycle;
    \draw[thick,fill,red] (1.32,1.45) -- (1.32, 2.7) -- (2.57, 2.7) -- (2.57,1.45) -- cycle;
    \draw[thick,fill,red] (2.62, -2.7) -- (2.62,-1.45) -- (3.25,-1.45) arc (-90:-73:1.45) -- +(-73:1.25) arc (-73:-90:2.7) -- cycle;
    \draw[thick,fill,red] (3.25,0) +(-36:1.45) arc(-36:0:1.45) -- +(0:1.25) arc (0:-36:2.7);
    \draw[thick,fill,red] (3.25,0) +(36:1.45) arc(36:71:1.45) -- +(71:1.25) arc (71:36:2.7);
    \draw[thick,fill=red!60] (-1.95,2.1) circle (0.6) node {\large \bf 7};
    \draw[thick,fill=blue!50] (-0.65,2.1) circle (0.6) node {\large \bf 4};
    \draw[thick,fill=red!60] (0.65,2.1) circle (0.6) node {\large \bf 5};
    \draw[thick,fill=blue!50] (1.95,2.1) circle (0.6) node {\large \bf 3};
    \draw[thick,fill=red!60] (3.25,2.1) circle (0.6) node {\large \bf 2};
    \draw[thick,fill=blue!50] (3.25,0) ++(54:2.1) circle (0.6) node {\large \bf 1};
    \draw[thick,fill=red!60] (3.25,0) ++(18:2.1) circle (0.6) node {\large \bf 4};
    \draw[thick,fill=blue!50] (3.25,0) ++(-18:2.1) circle (0.6) node {\large \bf 2};
    \draw[thick,fill=red!60] (3.25,0) ++(-54:2.1) circle (0.6) node {\large \bf 3};
    \draw[thick,fill=blue!50] (3.25,-2.1) circle (0.6) node {\large \bf 10};
    \draw[thick,fill=red!60] (1.95,-2.1) circle (0.6) node {\large \bf 1};
    \draw[thick,fill=blue!50] (0.65,-2.1) circle (0.6) node {\large \bf 5};
    \draw[thick,fill=red!60] (-0.65,-2.1) circle (0.6) node {\large \bf 6};
    \draw[thick,fill=blue!50] (-1.95,-2.1) circle (0.6) node {\large \bf 6};
    \draw[thick,fill=red!60] (-3.25,-2.1) circle (0.6) node {\large \bf 9};
    \draw[thick,fill=blue!50] (-3.25,0) ++(234:2.1) circle (0.6) node {\large \bf 8};
    \draw[thick,fill=red!60] (-3.25,0) ++(198:2.1) circle (0.6) node {\large \bf 8};
    \draw[thick,fill=blue!50] (-3.25,0) ++(162:2.1) circle (0.6) node {\large \bf 9};
    \draw[thick,fill=red!60] (-3.25,0) ++(126:2.1) circle (0.6) node {\large \bf 10};
    \draw[thick,fill=blue!50] (-3.25,2.1) circle (0.6) node {\large \bf 7};
    \draw[ultra thick,->] (6.5,0) -- (8,0);
\end{tikzpicture}
\begin{tikzpicture}[scale=0.4]
    \draw[thick,fill,lightgray] (-3.25,0) circle (3);
    \draw[thick,fill,lightgray] (3.25,0) circle (3);
    \draw[thick,fill,lightgray] (-3.5,-3) -- (3.5,-3) -- (3.5,3) -- (-3.5,3) -- (-3.5,-3);
    \draw[thick,fill,lightgray] (0.65,2) ellipse (3.6 and 2.9);
    \draw[fill=plum] (0.65,2) ellipse (3.2 and 2.5);
    \begin{scope}[shift={(0.65,0)},xscale=1.2,yscale=0.9]
    \draw (-0.1,1.35) -- (0.1,1.35) -- (0.1,-0.5) -- (-0.1,-0.5) -- (-0.1,1.35);
    \draw (-0.5,1.35) -- (-0.3,1.35) -- (-0.3,-0.45) -- (-0.5,-0.45) -- (-0.5,1.35);
    \draw (-0.9,1.35) -- (-0.7,1.35) -- (-0.7,-0.35) -- (-0.9,-0.35) -- (-0.9,1.35);
    \draw (-1.3,1.35) -- (-1.1,1.35) -- (-1.1,-0.15) -- (-1.3,-0.15) -- (-1.3,1.35);
    \draw (-1.7,1.35) -- (-1.5,1.35) -- (-1.5,0.15) -- (-1.7,0.15) -- (-1.7,1.35);
    \draw (-2.1,1.35) -- (-1.9,1.35) -- (-1.9,0.75) -- (-2.1,0.75) -- (-2.1,1.35);
    \draw (0.5,1.35) -- (0.3,1.35) -- (0.3,-0.45) -- (0.5,-0.45) -- (0.5,1.35);
    \draw (0.9,1.35) -- (0.7,1.35) -- (0.7,-0.35) -- (0.9,-0.35) -- (0.9,1.35);
    \draw (1.3,1.35) -- (1.1,1.35) -- (1.1,-0.15) -- (1.3,-0.15) -- (1.3,1.35);
    \draw (1.7,1.35) -- (1.5,1.35) -- (1.5,0.15) -- (1.7,0.15) -- (1.7,1.35);
    \draw (2.1,1.35) -- (1.9,1.35) -- (1.9,0.75) -- (2.1,0.75) -- (2.1,1.35);
    \end{scope}
    \draw (0.65,3.8) circle(.45);
    \draw[thick,fill,gray] (-3.25,2.7) arc (90:270:2.7) -- (-3.25,-1.45) arc (270:90:1.45) -- (-3.25,2.7);
    \draw[thick,fill,gray] (3.25,2.7) arc (90:-90:2.7) -- (3.25,-1.45) arc (-90:90:1.45) -- (3.25,2.7);
    \draw[thick,fill,gray] (-3.5,2.7) -- (3.5,2.7) -- (3.5,1.45) -- (-3.5,1.45) -- (-3.5,2.7);
    \draw[thick,fill,gray] (-3.5,-2.7) -- (3.5,-2.7) -- (3.5,-1.45) -- (-3.5,-1.45) -- (-3.5,-2.7);
    \draw[thick,fill=yellow] (-1.95,2.1) circle (0.6) node {\large \bf 14};
    \draw[thick,fill=yellow] (-0.65,2.1) circle (0.6) node {\large \bf 7};
    \draw[thick,fill=yellow] (0.65,2.1) circle (0.6) node {\large \bf 10};
    \draw[thick,fill=yellow] (1.95,2.1) circle (0.6) node {\large \bf 5};
    \draw[thick,fill=yellow] (3.25,2.1) circle (0.6) node {\large \bf 4};
    \draw[thick,fill=yellow] (3.25,0) ++(54:2.1) circle (0.6) node {\large \bf 1};
    \draw[thick,fill=yellow] (3.25,0) ++(18:2.1) circle (0.6) node {\large \bf 8};
    \draw[thick,fill=yellow] (3.25,0) ++(-18:2.1) circle (0.6) node {\large \bf 3};
    \draw[thick,fill=yellow] (3.25,0) ++(-54:2.1) circle (0.6) node {\large \bf 6};
    \draw[thick,fill=yellow] (3.25,-2.1) circle (0.6) node {\large \bf 19};
    \draw[thick,fill=yellow] (1.95,-2.1) circle (0.6) node {\large \bf 2};
    \draw[thick,fill=yellow] (0.65,-2.1) circle (0.6) node {\large \bf 9};
    \draw[thick,fill=yellow] (-0.65,-2.1) circle (0.6) node {\large \bf 12};
    \draw[thick,fill=yellow] (-1.95,-2.1) circle (0.6) node {\large \bf 11};
    \draw[thick,fill=yellow] (-3.25,-2.1) circle (0.6) node {\large \bf 18};
    \draw[thick,fill=yellow] (-3.25,0) ++(234:2.1) circle (0.6) node {\large \bf 15};
    \draw[thick,fill=yellow] (-3.25,0) ++(198:2.1) circle (0.6) node {\large \bf 16};
    \draw[thick,fill=yellow] (-3.25,0) ++(162:2.1) circle (0.6) node {\large \bf 17};
    \draw[thick,fill=yellow] (-3.25,0) ++(126:2.1) circle (0.6) node {\large \bf 20};
    \draw[thick,fill=yellow] (-3.25,2.1) circle (0.6) node {\large \bf 13};
\end{tikzpicture}
\caption{A possible fix for a puzzle with $n=20$ and $k=5$.}
\label{fig:collection4}
\end{figure}

To compare, for collection (5) we actually need an odd number of cycles for each individual pile/color.  The smallest such puzzle is one with $n=14$ and $k=9$ where we could make odds red and evens blue (thus building a Type I permutation since we did not swap the colors).  The red permutation $(1~4~2~5)(3~7)(6)$ together with the blue permutation $(1~3~2~4~7~6~5)$ corresponds to the solvable puzzle below.
$$\left(\circled{9}~\circled{10}~\circled{7}~\circled{6}~\circled{13}~\circled{2}~\circled{1}~\circled{4}~\circled{3}\right)\circled{12}~ \circled{11}~\circled{14}~\circled{5}~\circled{8}$$


\end{document}